\documentclass[12pt]{article}
\usepackage{amsmath,amsthm}
\usepackage{amsfonts,mathrsfs}
\usepackage{color}
\usepackage[colorlinks,anchorcolor=blue,citecolor=blue]{hyperref}
\usepackage{natbib}

\allowdisplaybreaks
\theoremstyle{plain}
\newtheorem{thm}{\noindent\bf Theorem}
\newtheorem{cor}{\noindent\bf Corollary}
\newtheorem{lem}{\noindent\bf Lemma}
\newtheorem{prop}{\noindent\bf Proposition}
\theoremstyle{remark}
\newtheorem{rmk}{\noindent \bf \textit{Remark}}
\newtheorem{exam}{\noindent \bf Example}
\newcommand{\D}{\displaystyle}
\newcommand{\non}{\nonumber}
\newcommand{\bP}{\mathbb{P}}
\newcommand{\bE}{\mathbb{E}}

\newcommand{\wo}{W^{(\omega)}}
\newcommand{\zo}{Z^{(\omega)}}
\newcommand{\zoh}{\widehat{Z}^{(\omega)}}

\newcommand{\ro}{R^{(\omega)}}

\newcommand{\ho}{H^{(\omega)}}

\newcommand{\woq}{W^{(\omega+\delta)}}
\newcommand{\hoq}{H^{(\omega+\delta)}}

\begin{document}

%
%

\title{How long does the surplus stay close to its historical  high?}
\author{ Bo Li, Yun Hua and Xiaowen Zhou}

\AtEndDocument{\bigskip{\footnotesize%
  \textsc{School of Mathematics and LPMC, Nankai University, Tianjin, P.R.China} \par
  \textit{E-mail address}, Bo Li: \texttt{libo@nankai.edu.cn} \par
  \addvspace{\medskipamount}
  \textsc{School of Mathematics and LPMC, Nankai University, Tianjin, P.R.China} \par
  \textit{E-mail address}, Yun Hua: \texttt{Huayun@mail.nankai.edu.cn} \par
  \addvspace{\medskipamount}
  \textsc{Department of Mathematics and Statistics, Concordia University} \par
  \textit{E-mail address}, Xiaowen Zhou: \texttt{xiaowen.zhou@concordia.ca}
}}

\date{}

\maketitle

\begin{abstract}
In this paper we find the Laplace transforms of the weighted occupation times for a spectrally negative L\'evy surplus process to spend below its running maximum up to the first exit times. The results are expressed in terms of generalized  scale functions. For step weight functions, the Laplace transforms can be further expressed in terms of scale functions.
\end{abstract}
%

\section{Introduction}\label{intro}

In risk theory, in addition to ruin behaviors, the surplus process's behaviors before ruin are also of interest.  Intuitively, it is an ideal situation if the surplus has typically small downward fluctuations from its historical   high.   It is thus interesting to know how long the surplus stays near its running maximum up to certain times. The amount of such time can serve  as  a criterion  for performance of the surplus process. We are not aware of any previous work along this line, which
 motivates our study of occupation times associated to the running maximum of a spectrally negative L\'evy process.

Since the work of \cite{Landriault2011:occupationtime:levy}, the occupation times  have been studied with different approaches in a series of papers for spectrally negative L\'evy processes and related processes such as the reflected or refracted spectrally negative L\'evy processes; see \cite{ZhouBo:localtime} for a summary of these results. So far all of these results concern the occupation times over deterministic intervals.
It is therefore also interesting, from the point of view of occupation time theory, to study the occupation time  over an interval with random boundaries. To our best knowledge, such occupation times have not been studied for spectrally negative L\'evy processes before, and our work represents the first attempt in this direction.

In this paper we consider the Laplace transforms of weighted occupation times the spectrally negative L\'evy process spends near its running maximum.  By considering the time it spends below the running maximum,
we can relate the problem with a problem on reflected L\'evy process. To this end we modify the approach of
\cite{BoLi:sub1}.  Taking use of Feynman-Kac type equations, see e.g. III.19 of \cite{Rogers:book:v1}, we express  the desired Laplace transforms using a generalized scale function, which is the unique solution to an integral equation involving the scale function and the weight function. Similar arguments to obtain Laplace transforms of weighted occupation times for  refracted spectrally negative L\'evy processes can be found in \cite{LiZh18}.
  When the weight function is a step function, the generalized scale function can be further expressed in terms of iterated integrals of the scale functions for spectrally negative L\'evy processes.

This paper is structured as follows. After the introduction in Section \ref{intro}, the main results and examples are presented in Section \ref{main}. Proofs of the main results are deferred to Section \ref{proof}.

\section{Main results}\label{main}

Let $X$ be a spectrally negative L\'{e}vy process,
$S:=\{S_{t}:=\sup_{s\in[0,t]}X_{s}, t\geq0\}$ be the running maximum process of $X$,
and $Y:=S-X=\{S_{t}-X_{t}, t\geq0\}$ be the reflected process from the running  maximum.
In this paper, we are interested in the
Laplace transforms of weighted occupation times of $Y$ up to the first passage times of $X$.
More precisely,
let $\omega$ be a nonnegative, locally bounded measurable function on $\mathbb{R}$,
we want to identify expressions of Laplace transforms
\[
\bE_{x}\big(e^{-L(\tau_{b}^{+})}; \tau_{b}^{+}<\tau_{0}^{-}\big)
\quad\text{and}\quad
 \bE_{x}\big(e^{-L(\tau_{0}^{-})}; \tau_{0}^{-}<\tau_{b}^{+}\big),
\]
where $\D L(t):=\int_{0}^{t}\omega(Y_{s})\,ds$ denotes the $\omega$-weighted occupation time near $S$,
and where the first passage times of $X$ are defined by
\begin{equation}
\tau_{x}^{+}:=\inf\{t>0, X_{t}>x\}
\quad\text{and}\quad
\tau_{x}^{-}:=\inf\{t>0, X_{t}<x\}
\end{equation}
with the convention $\inf\emptyset:=\infty$.
Notice that, for the case of $\omega(z)=\mathbf{1}(z\in[a,b))$ for some $b>a>0$,
$$\D L(t)=\int_{0}^{t}\mathbf{1}(S_{r}-b<X_{r}\leq S_{r}-a)\,dr$$
 is the occupation time process $X$ spends in a random interval below its maximum process $S$, which is also studied in examples at the end of this section.

In the fluctuation theory for L\'evy processes, the interested quantities are often expressed using scale functions $W^{(q)}$ and $Z^{(q)}$, where for $q\geq 0$, $W^{(q)}$ is a nonnegative and increasing function satisfying $W^{(q)}(x)=0$ for $x<0$ and
\[\int_0^\infty W^{(q)}(x)e^{-\lambda x}dx=\frac{1}{\psi(\lambda)-q}, \,\, \lambda>\Phi(q), \]
and
\[\D Z^{(q)}(x):=1+ q\int_{0}^{x}W^{(q)}(y)\,dy.\]
Here $\psi(\lambda):=\log \bE e^{\lambda X_1}$ for $\lambda\geq 0$ denotes the Laplace exponent for $X$ and  $\Phi(q):=\inf\{\lambda>0, \psi(\lambda)>q\}$ denotes the right inverse function of $\psi$. In addition,
\[\psi(\lambda)=\mu\lambda+\frac{1}{2}\sigma^2\lambda^2+\int_{(0,\infty)}(e^{-\lambda x}-1+\lambda x\mathbf{1}_{(x<1)})\Pi(dx),\]
where the $\sigma$-finite L\'evy measure $\Pi$ on $(0, \infty)$ satisfies $\int_{(0, \infty)} 1\wedge x^2 \Pi(dx)<\infty$.
Write $W$ and $Z$ for $W^{(0)}$ and $Z^{(0)}$, respectively. We refer to \cite{Kyprianou2014:book:levy} for more detailed introduction on scale functions.

To express our results, we need the so called $\omega$-scale function first introduced in \cite{BoLi:sub1},
which is defined as the unique locally bounded function satisfying the following equation
\begin{equation}\label{eqn:defn:w:1}
\wo(x,y)=W(x-y)+ \int_{y}^{x}W(x-z)\omega(z)\wo(z,y)\,dz, \,\, x, y\in \mathbb{R}.
\end{equation}
The $\omega$-scale function is further studied in \cite{ZhouBo:localtime} and shown to satisfy a dual version of the above equation.
\begin{equation}\label{eqn:defn:w:2}
\wo(x,y)=W(x-y)+\int_{y}^{x}\wo(x,z)\omega(z)W(z-y)\,dz;
\end{equation}
see Lemma 2 of \cite{ZhouBo:localtime}.

Throughout this paper we always assume that the derivative $W^{\prime}(x)$ is continuous  for $x\in (0, \infty)$, which holds if process $X$ has a Brownian component or the L\'evy measure allows a density.
We refer the readers to \cite{chan2011:scalefunction:smooth} for more detailed discussion on the smoothness of scale functions.
For convenience we also assume that function $\omega$ is  right continuous with left limit.

For $x>y$ denote by
\[
\begin{split}
\wo_{1}(x,y)
:=& \lim_{z\to x+}\frac{\wo(z,y)-\wo(x,y)}{z-x}\\
\text{and}\quad
\wo_{2}(x,y)
:=& \lim_{z\to y+}\frac{\wo(x,z)-\wo(x,y)}{z-y}
\end{split}
\]
the right partial derivatives of $\wo(x,y)$ on  $x$ and $y$, respectively.
One can check that
\begin{align}
\wo_{1}(x,y)=&\ W'(x-y)+ \int_{y}^{x}W'(x-z)\omega(z)\wo(z,y)\,dz+
W(0)\omega(x)\wo(x,y),\non\\
\label{eqn:wy}
\wo_{2}(x,y)=&\ -W'(x-y)- \int_{y-}^{x}\wo(x,z)\omega(z)W(dz-y)
\end{align}
where $W(dz)$ is the Stieltjes measure induced by $W$ with $W(\{0\})=W(0)$ and it is known that $W(0)>0$ if process $X$ has sample paths of bounded variation.

Given $\wo$ we also need to introduce function
\begin{equation}\label{eqn:defn:zh}
\zoh(x,y):=1+\int_{y}^{x}\omega(z)\wo(z,y)\,dz
\end{equation}
for $x, y\in
\mathbb{R}$.
The right partial derivatives of $\zoh$ are given by
\[\zoh_{1}(x,y):= \lim_{z\rightarrow x+}\frac{\zoh(z,y)-\zoh(x,y)}{z-x} =\omega(x)\wo(x,y)\]
and
\[\zoh_{2}(x,y):=  \lim_{z\rightarrow y+}\frac{\zoh(x,z)-\zoh(x,y)}{z-y}
=\int_{y}^{x}\omega(z)\wo_{2}(z,y)\,dz-\omega(y)W(0).\]

\begin{rmk}
We remark that, from  equations \eqref{eqn:defn:w:1} and \eqref{eqn:defn:w:2},
$\wo$ is absolutely continuous with respect to Lebesgue measure in both $x$ and $y$.
But  it may fail to be differentiable for general weight function $\omega$.
However, since $\omega$ is assumed to be right-continuous with left limit, both
the right derivative and  left derivative of $\wo(x,\cdot)$  exist and
the integral equation \eqref{eqn:wy} can be proved directly for $x>y$.
Moreover, $\wo_{2}(x,\cdot)$ is continuous at $y_{0}$ if and only if $\omega$ is continuous at $y_{0}$,
and in such a case, $\wo(x,\cdot)$ is differentiable at $y_{0}$. The same is true for $\wo_{1}$, $\zo_{1}$ and $\zo_{2}$.
\end{rmk}

For simplicity we write $\wo(x)=\wo(x,0)$, $\zoh(x)=\zoh(x,0)$ and write
\[
\wo_{1}(x):=\wo_{1}(x,0)=W'(x)+ \int_{0}^{x}W'(x-z)\omega(z)\wo(z)\,dz+ W(0)\omega(x)\wo(x),
\]
\begin{equation}\label{wo_2}
\wo_{2}(x):=\ \wo_{2}(x,0)= -W'(x)- \int_{0-}^{x}\wo(x,z)\omega(z)W(dz),
\end{equation}
\[\zoh_{1}(x):=\ \zoh_{1}(x,0)=\omega(x)\wo(x),\]
\[\zoh_{2}(x):=\ \zoh_{2}(x,0)=\int_{0}^{x}\omega(z)\wo_{2}(z)\,dz-\omega(0)W(0).\]

Define for $u>0$
\begin{equation}
\label{defn:ho}
\ho(u):=\exp\Big(-\int_{1}^{u} \frac{\wo_{2}(z)}{\wo(z)}\,dz\Big),
\end{equation}
where for $0<u<1$ the integral $\int_1^u $ is understood as $-\int_u^1 $.

We first present a result on potential density. Note that the exit time involved is for process $X$.

\begin{prop}\label{lem:3}
For any $0\leq x<z<b$ and $z>y>0$, we have
\begin{align}\label{eqn:thm1:ro+}
 \ro(x;dz,dy)&:=
\int_{0}^{\infty} \bE_{x}\Big(e^{-L(t)}; S_{t}\in dz, Y_{t}\in dy, t<\tau_{b}^{+}\wedge\tau_{0}^{-}\Big)\,dt \non\\
&= \frac{\ho(x)}{\ho(z)}
\Big(\frac{\wo_{2}(z)}{\wo(z)}\wo(y)- \wo_{2}(y)\Big)\,dy\,dz
\end{align}
and $$\D\ro(x;dz,\{0\})= \frac{\ho(x)}{\ho(z)} W(0)\,dz.$$
\end{prop}

\begin{thm}\label{thm:1} For any $b>0$ and $x\in[0,b]$, we have
\begin{equation}\label{eqn:thm1:b+}
\bE_{x}\Big(\exp\Big(-\int_{0}^{\tau_{b}^{+}} \omega(Y_{t})\,dt\Big);
\tau_{b}^{+}<\tau_{0}^{-}\Big)=
\frac{\ho(x)}{\ho(b)},
\end{equation}
and
\begin{equation}\label{eqn:thm1:c-}
\begin{split}
&\ \bE_{x}\Big(\exp\Big(-\int_{0}^{\tau_{0}^{-}} \omega(Y_{t})\,dt\Big);
\tau_{0}^{-}<\tau_{b}^{+}\Big)\\
=&\ \zoh(x)- \frac{\ho(x)}{\ho(b)}\zoh(b)
+ \int_{x}^{b} \frac{\ho(x)}{\ho(z)}\big(\zoh_{1}(z)+\zoh_{2}(z)\big)\,dz.
\end{split}
\end{equation}
\end{thm}

\begin{rmk}
We remark that for a locally bounded measurable and nonnegative function $\omega$, the partial derivatives   of $\wo(\cdot,\cdot)$ and $\zoh(\cdot,\cdot) $  exist Lebesgue a.e. Then the results in Theorem \ref{thm:1} still hold with the right partial derivatives replaced with the respective partial derivatives.
\end{rmk}

If $\omega\equiv q$ for some $q>0$,
then $\wo(x,y)=W^{(q)}(x-y)$ and one can check that $\D \ho(u)=\frac{W^{(q)}(u)}{W^{(q)}(1)}$.
Therefore, expression \eqref{eqn:thm1:b+} reduces to the classical Laplace transform for the two-sided passage problem.
On the other hand,
\begin{align*}
&\ \zoh(x,y)= 1+q \int_{y}^{x}W^{(q)}(z-y)\,dz=Z^{(q)}(x-y),\\
&\ \zoh(x)= Z^{(q)}(x)\quad\text{and}\quad\zoh_{1}(x,y)+\zoh_{2}(x,y)=0,
\end{align*}
and \eqref{eqn:thm1:c-} also coincides with the corresponding classical result for the two-sided exit problem.

For the two-sided exit problems, one may also be interested in the time
$\tau_{c}^{-}$ instead of $\tau_{0}^{-}$  for some $c<b$.
Since $X$ is spatially homogenous, with a shift argument applied,
we have following results with general initial value.

\begin{cor} For any $x\in[c,b]$, we have
\begin{equation}
\bE_{x}\Big(\exp\big(-\int_{0}^{\tau_{b}^{+}} \omega(Y_{t})\,dt\big);
\tau_{b}^{+}<\tau_{c}^{-}\Big)=
\frac{\ho(x-c)}{\ho(b-c)},
\end{equation}
and
\begin{equation}
\begin{split}
&\ \bE_{x}\Big(\exp\big(-\int_{0}^{\tau_{c}^{-}} \omega(Y_{t})\,dt\big);
\tau_{c}^{-}<\tau_{b}^{+}\Big)\\
=&\ \zoh(x-c)- \frac{\ho(x-c)}{\ho(b-c)}\zoh(b-c)\\
&\quad + \int_{x}^{b} \frac{\ho(x-c)}{\ho(z-c)}\big(\zoh_{1}(z-c)+\zoh_{2}(z-c)\big)\,dz.
\end{split}
\end{equation}
\end{cor}

The Gerber-Shiu penalty function is of great interest in ruin theory,
which describes  the joint distribution of
the time of ruin, the surplus before ruin as well as the deficit at ruin.
It has been generalized to different forms.
With the occupation time near the running maximum taken into consideration,
we  have following version of Gerber-Shiu function.

\begin{prop}[Gerber-Shiu function]\label{prop:3}
For any $x,z\in[0,b]$ and $y>0$, we have
\begin{equation}\label{eqn:prop3:1}
\begin{split}
&\ \bE_{x}\big(e^{-\delta \tau_{0}^{-}-L(\tau_{0}^{-})};
X(\tau_{0}^{-}-)\in\,dz, |X(\tau_{0}^{-})|\in\,dy, S(\tau_{0}^{-})<b\big)\\
=&\ \left(\frac{\hoq(x)}{\hoq(b)}\woq(b-z)- \woq(x-z)\right) \Pi(dy+z)\,dz\\
&\quad -\Big(\int_{x}^{b} \frac{\hoq(x)}{\hoq(u)}
\big(\woq_{1}(u-z)+\woq_{2}(u-z)\big)\,du\Big) \Pi(dy+z)\,dz.
\end{split}
\end{equation}
\end{prop}
We can also find Laplace transform of the occupation time involving  the creeping event that occurs
when $X$ first exits  a lower level by hitting the level with positive probability,
 which is also called ruin caused by oscillation in ruin theory.
But more notations are needed and we omit it.

If $\omega$ is a step function,
a recursive expression for $\wo(x,y)$ in terms of scale function is given in (3.24) of \cite{BoLi:sub1};
 also see similar results in \cite{KuZh17} and \cite{LiZh18}.
\textcolor{red}{In this paper we present a similar result for  $\zoh(x,y)$ for $x,y\in\mathbb{R}$,
and an alternative recursive expression for $\wo(x,y)$.}
Note that $\zoh(x,y)=1$ for $x\leq y$ by definition.

Let $\{a_{k}\}_{k\geq1}$ with $a_{k+1}<a_{k}$ be a partition of $\mathbb{R}$,
and let $\{p_{k}\}_{k\geq0}$ be a sequence of nonnegative constants. Define $\omega_{0}(x):=p_{0}$ and for $n\geq1$
\begin{equation}\label{step_function}
\omega_{n}(x):= p_{0} \mathbf{1}(x\geq a_{1}) + \sum_{k=1}^{n-1} p_{k} \mathbf{1}(a_{k+1}\leq x<a_{k})
+ p_{n} \mathbf{1}(x< a_{n}).
\end{equation}
Denote by
$W_{n}(x,y):=W^{(\omega_{n})}(x,y)$ and
$\widehat{Z}_{n}(x,y):=\widehat{Z}^{(\omega_{n})}(x,y)$
the scale functions with respect to $\omega_{n}$.

\begin{prop}\label{prop:2}
$W_{0}(x,y)=W^{(p_{0})}(x-y)$ and for $n\geq 0$,
\[
W_{n+1}(x,y)=W_{n}(x,y)+ (p_{n+1}-p_{n})\int_{y}^{a_{n+1}}
W_{n}(x,z)W^{(p_{n+1})}(z-y)\,dz.
\]

$\widehat{Z}_{0}(x,y)=Z^{(p_{0})}(x-y)$ and for $n\geq 0$,
\[
\widehat{Z}_{n+1}(x,y)
=\widehat{Z}_{n}(x,y)+ (p_{n+1}-p_{n})\int_{y}^{a_{n+1}}
\widehat{Z}_{n}(x,z)W^{(p_{n+1})}(z-y)\,dz.
\]
\end{prop}
Noticing that for  $y>a_{n+1}$,
we have $\widehat{Z}_{n+1}(x,y)=\widehat{Z}_{n}(x,y)$ from the equation above,
since $W^{(p_{n+1})}(z-y)=0$ for $z\in(a_{n+1},y)$,
which can be observed from \eqref{eqn:prop:1} and used in our proofs.

Considering a special case that
\begin{equation}\label{step_function_a}
\omega(z)=p+ (q-p) \mathbf{1}(a_{2}\leq z<a_{1})
\end{equation}
 for some $q\geq p\geq 0$ and $0\leq a_{2}<a_{1}$,
which is the weight function first considered in \cite{Loeffen2014:occupationtime:levy} and \cite{Zhou2014:occupationtime:levy},
where the occupation time of intervals of spectrally negative L\'{e}vy process is studied and
 the following auxiliary functions are introduced.
\begin{align*}
W^{(p,q)}_{(a_{2})}(x):=&\ W^{(p)}(x)+(q-p)\int_{a_{2}}^{x}W^{(q)}(x-z)W^{(p)}(z)\,dz\\
=&\ W^{(q)}(x)- (q-p) \int_{0}^{a_{2}}W^{(q)}(x-z)W^{(p)}(z)\,dz,\\
W^{(p,q,p)}_{(a_{2},a_{1})}(x):=&\ W^{(p,q)}_{(a_{2})}(x)+ (p-q) \int_{a_{1}}^{x}W^{(p)}(x-z)W^{(p,q)}_{(a_{2})}(z)\,dz.
\end{align*}
It was pointed out  in \cite{BoLi:sub1} that
$W^{(p,q,p)}_{(a_{2},a_{1})}(x)=\wo(x,0)$ for the function $\omega$ defined above.

Note that $W^{(\omega_{u})}(u-y,u-x)=\wo(x,y)$ for every $u\in\mathbb{R}$ by Lemma \ref{lem:1},
where
\begin{equation*}
\begin{split}
\omega_{u}(z)&=\omega(u-z)=p+ (q-p) \mathbf{1}(a_{2}<u-z<a_{1})\\
&=p+ (q-p) \mathbf{1}(u-a_{1}<z<u-a_{2}).
\end{split}
\end{equation*}
Since
$\D \wo_{2}(x,y)=-W^{(\omega_{u})}_{1}(u-y,u-x)$,
taking $x=u=t$ and $y=0$ we have
\[
\frac{\wo_{2}(t)}{\wo(t)}
=\frac{\wo_{2}(t,0)}{\wo(t,0)}
=- \frac{W_{1}^{(\omega_{t})}(t,0)}{W^{(\omega_{t})}(t,0)}
=- \frac{W_{(t-a_{1},t-a_{2})}^{(p,q,p)\prime}(t)}{W_{(t-a_{1},t-a_{2})}^{(p,q,p)}(t)}.
\]
Therefore, we obtain the following alternative expression for the function $\omega$ given in (\ref{step_function_a}).
\[
\bE_{x}\Big(\exp\big(-\int_{0}^{\tau_{b}^{+}} \omega(Y_{t})\,dt\big);
\tau_{b}^{+}<\tau_{0}^{-}\Big)=
\exp\Big( -\int_{x}^{b}
\frac{W_{(t-a_{2},t-a_{1})}^{(p,q,p)\prime}(t)}{W_{(t-a_{2},t-a_{1})}^{(p,q,p)}(t)}\,dt
\Big).
\]
In addition,  by Proposition \ref{prop:2} function $\zoh\equiv\widehat{Z}^{(p,q,p)}_{(a_{2},a_{1})}$ for the above mentioned weight function $\omega$ is given by
\[\widehat{Z}^{(q,p)}_{(a_{1})}(x,y):=
Z^{(p)}(x-y)+ (q-p)\int_{y}^{a_{1}} Z^{(p)}(x-z)W^{(q)}(z-y)\,dz\]
 and
\[\widehat{Z}^{(p,q,p)}_{(a_{2},a_{1})}(x,y):=
\widehat{Z}^{(q,p)}_{(a_{1})}(x,y)+ (p-q)\int_{y}^{a_{2}} \widehat{Z}^{(q,p)}_{(a_{1})}(x,z)W^{(p)}(z-y)\,dz.\]
Then a more explicit expression for \eqref{eqn:thm1:c-} also follows.

To further simplify the expression,
let $a_{1}=a$ and $a_{2}=0$ for some $a>0$,
that is
\[\omega(z)=q\mathbf{1}(0\leq z<a)+p \mathbf{1}(z\geq a), \,\,\, z\geq0.\]
We have by definitions that for $t\geq0$,
\begin{align*}
W_{(t-a)}^{(p,q)}(t)
=&\ W^{(p)}(x)+ (q-p) \int_{t-a}^{t}W^{(q)}(t-z)W^{(p)}(z)\,dz\\
=&\ W^{(p)}(x)+ (q-p)  \int_{0}^{a} W^{(p)}(t-s)W^{(q)}(s)\,ds
=W_{(a)}^{(q,p)}(t)\\
=&\ W^{(q)}(x)+ (p-q) \int_{a}^{t} W^{(p)}(t-s)W^{(q)}(s)\,ds,\\
\widehat{Z}^{(q,p)}_{(a)}(t)=&\ Z^{(p)}(t)+ (q-p)\int_{0}^{a}Z^{(p)}(t-s)W^{(q)}(s)\,ds\\
=&\ Z^{(q)}(t)+ (p-q)\int_{a}^{t}Z^{(p)}(t-s)W^{(q)}(s)\,ds.
\end{align*}

\begin{exam}
Let $X=\mu t+\sigma B_{t}$ be a Brownian surplus process,
where $\mu$ and $\sigma>0$ are constants and $B_{t}$ is a standard Brownian motion.
It is known that for $x>0$
\[
\begin{aligned}
W^{(q)}(x)=&\ \frac{e^{\rho_{1}x}- e^{\rho_{2}x}}{D\cdot (\rho_{1}-\rho_{2})}
& \text{and}&&
W^{(p)}(x)=&\ \frac{e^{\eta_{1}x}- e^{\eta_{2}x}}{D\cdot (\eta_{1}-\eta_{2})},\\
Z^{(q)}(x)=&\ \frac{\rho_{2}e^{\rho_{1}x}-\rho_{1} e^{\rho_{2}x}}{\rho_{2}-\rho_{1}}
&\text{and}&&
Z^{(p)}(x)=&\ \frac{\eta_{2}e^{\eta_{1}x}-\eta_{1} e^{\eta_{2}x}}{\eta_{2}-\eta_{1}},
\end{aligned}
\]
where $D=\sigma^{2}/2$,
$\rho_{1}>\rho_{2}$ are the two roots to equation $D s^{2}+ \mu s=q$
and $\eta_{1}>\eta_{2}$ are the two roots to equation $D s^{2}+ \mu s=p$.
We thus have
\begin{align*}
&\ W_{(t-a)}^{(p,q)}(t)=W^{(q,p)}_{(a)}(t)
=W^{(q)}(t)+ (p-q)\int_{a}^{t}W^{(p)}(t-s)W^{(q)}(s)\,ds\\
=&\ \frac{e^{\rho_{1}t}- e^{\rho_{2}t}}{D\cdot (\rho_{1}-\rho_{2})}+
\frac{(p-q)\mathbf{1}(t>a)}{D^{2}(\rho_{1}-\rho_{2})(\eta_{1}-\eta_{2})}
\int_{a}^{t}
\big(e^{\eta_{1}(t-s)}-e^{\eta_{2}(t-s)}\big)
\big(e^{\rho_{1}s}-e^{\rho_{2}s}\big)\,ds\\
=&\ \frac{e^{\rho_{1}t}- e^{\rho_{2}t}}{D\cdot (\rho_{1}-\rho_{2})}+
\frac{(p-q)\mathbf{1}(t>a)}{D^{2}(\rho_{1}-\rho_{2})(\eta_{1}-\eta_{2})}
\Big(\frac{e^{\rho_{1}t}-e^{\rho_{1}a+\eta_{1}(t-a)}}{\rho_{1}-\eta_{1}}
+ \frac{e^{\rho_{2}t}-e^{\rho_{2}a+\eta_{1}(t-a)}}{\eta_{1}-\rho_{2}}\\
&\quad+ \frac{e^{\rho_{1}t}-e^{\rho_{1}a+\eta_{2}(t-a)}}{\eta_{2}-\rho_{1}}
+  \frac{e^{\rho_{2}t}-e^{\rho_{2}a+\eta_{2}(t-a)}}{\rho_{2}-\eta_{2}}\Big).
\end{align*}
Similarly, we have the following expression for $\zoh(t)$.
\begin{align*}
&\ \widehat{Z}_{(a)}^{(q,p)}(t)=
Z^{(q)}(t)+ (p-q)\int_{a}^{t} Z^{(p)}(t-s)W^{(q)}(s)\,ds\\
=&\  \frac{\rho_{2}e^{\rho_{1}t}-\rho_{1} e^{\rho_{2}t}}{\rho_{2}-\rho_{1}}
+ \frac{(p-q)\mathbf{1}(t>a)}{D(\eta_{2}-\eta_{1})(\rho_{1}-\rho_{2})}
\int_{a}^{t}
(\eta_{2}e^{\eta_{1}(t-s)}- \eta_{1} e^{\eta_{2}(t-s)})
(e^{\rho_{1}s}-e^{\rho_{2}s})\,ds\\
=&\ \frac{\rho_{2}e^{\rho_{1}t}-\rho_{1} e^{\rho_{2}t}}{\rho_{2}-\rho_{1}}
+ \frac{(p-q)\mathbf{1}(t>a)}{D(\eta_{2}-\eta_{1})(\rho_{1}-\rho_{2})}
\Big(\frac{\eta_{2}(e^{\rho_{1}t}-e^{\rho_{1}a+\eta_{1}(t-a)})}{\rho_{1}-\eta_{1}}
+ \frac{\eta_{2}(e^{\rho_{2}t}-e^{\rho_{2}a+\eta_{1}(t-a)})}{\eta_{1}-\rho_{2}}\\
&\quad+ \frac{\eta_{1}(e^{\rho_{1}t}-e^{\rho_{1}a+\eta_{2}(t-a)})}{\eta_{2}-\rho_{1}}
+  \frac{\eta_{1}(e^{\rho_{2}t}-e^{\rho_{2}a+\eta_{2}(t-a)})}{\rho_{2}-\eta_{2}}\Big).
\end{align*}
\end{exam}


\begin{exam}
Let $X$ be a Cram\'er-Lundberg surplus process with exponentially distributed claim sizes,
that is, $\D \psi(s)=\mu s- \frac{\lambda s}{s+\beta}$ for some constants $\mu, \lambda, \beta>0$.
Then we have for $x>0$
\[
\begin{aligned}
W^{(q)}(x)=&\ \frac{\beta+\rho_{1}}{\mu(\rho_{1}-\rho_{2})}
e^{\rho_{1}x}+ \frac{\beta+\rho_{2}}{\mu(\rho_{2}-\rho_{1})} e^{\rho_{2}x},\\
W^{(p)}(x)=&\ \frac{\beta+\eta_{1}}{\mu(\eta_{1}-\eta_{2})}
e^{\eta_{1}x}+ \frac{\beta+\eta_{2}}{\mu(\eta_{2}-\eta_{1})} e^{\eta_{2}x},\\
Z^{(q)}(x)=&\ \frac{\rho_{2}(\beta+\rho_{1})}{\beta(\rho_{2}-\rho_{1})}
e^{\rho_{1}x}+ \frac{\rho_{1}(\beta+\rho_{2})}{\beta(\rho_{1}-\rho_{2})} e^{\rho_{2}x},\\
Z^{(p)}(x)=&\ \frac{\eta_{2}(\beta+\eta_{1})}{\beta(\eta_{2}-\eta_{1})}
e^{\eta_{1}x}+ \frac{\eta_{1}(\beta+\eta_{2})}{\beta(\eta_{1}-\eta_{2})} e^{\eta_{2}x},
\end{aligned}
\]
where $\rho_{1}>\rho_{2}$ are the roots to equation
$\mu s^{2}+ (\mu\beta-\lambda-q)s- q\beta=0$
and $\eta_{1}>\eta_{2}$ are the roots to equation
$\mu s^{2}+ (\mu\beta-\lambda-p)s- p\beta=0$.
We thus have
\begin{align*}
&\ W^{(q,p)}_{(a)}(t)=W^{(q)}(t)+
(p-q)\int_{a}^{t} W^{(p)}(t-s)W^{(q)}(s)\,ds\\
=&\ \frac{(\beta+\rho_{1})e^{\rho_{1}t}}{\mu(\rho_{1}-\rho_{2})}
+ \frac{(\beta+\rho_{2})e^{\rho_{2}t}}{\mu(\rho_{2}-\rho_{1})}
+ \frac{(p-q)\mathbf{1}(t>a)}{\mu^{2}(\rho_{1}-\rho_{2})(\eta_{1}-\eta_{2})}\\
&\quad\times
\Big(\frac{e^{\rho_{1}t}-e^{\rho_{1}a+\eta_{1}(t-a)}}{\rho_{1}-\eta_{1}}
(\beta+\eta_{1})(\beta+\rho_{1})
+ \frac{e^{\rho_{2}t}-e^{\rho_{2}a+\eta_{1}(t-a)}}{\eta_{1}-\rho_{2}}
(\beta+\eta_{1})(\beta+\rho_{2})\\
&\quad\quad+ \frac{e^{\rho_{1}t}-e^{\rho_{1}a+\eta_{2}(t-a)}}{\eta_{2}-\rho_{1}}
(\beta+\eta_{2})(\beta+\rho_{1})
+  \frac{e^{\rho_{2}t}-e^{\rho_{2}a+\eta_{2}(t-a)}}{\rho_{2}-\eta_{2}}
(\beta+\eta_{2})(\beta+\rho_{2})\Big).
\end{align*}
Similarly, we can also obtain the following expression for $\zoh(t)$.
\begin{align*}
&\ \widehat{Z}_{(a)}^{(q,p)}(t)=Z^{(q)}(t)+ (p-q)\int_{a}^{t}Z^{(p)}(t-s)W^{(q)}(s)\,ds\\
=&\ \frac{\rho_{2}(\beta+\rho_{1})}{\beta(\rho_{2}-\rho_{1})}
e^{\rho_{1}t}+ \frac{\rho_{1}(\beta+\rho_{2})}{\beta(\rho_{1}-\rho_{2})} e^{\rho_{2}t}
+ \frac{(q-p)\mathbf{1}(t>a)}{\mu\beta(\eta_{1}-\eta_{2})(\rho_{1}-\rho_{2})}\\
&\quad\times
\Big(\frac{e^{\rho_{1}t}-e^{\rho_{1}a+\eta_{1}(t-a)}}{\rho_{1}-\eta_{1}}
\eta_{2}(\beta+\eta_{1})(\beta+\rho_{1})
+ \frac{e^{\rho_{2}t}-e^{\rho_{2}a+\eta_{1}(t-a)}}{\eta_{1}-\rho_{2}}
\eta_{2}(\beta+\eta_{1})(\beta+\rho_{2})\\
&\quad\quad+ \frac{e^{\rho_{1}t}-e^{\rho_{1}a+\eta_{2}(t-a)}}{\eta_{2}-\rho_{1}}
\eta_{1}(\beta+\eta_{2})(\beta+\rho_{1})
+  \frac{e^{\rho_{2}t}-e^{\rho_{2}a+\eta_{2}(t-a)}}{\rho_{2}-\eta_{2}}
\eta_{1}(\beta+\eta_{2})(\beta+\rho_{2})\Big).
\end{align*}
\end{exam}

\section{Proof}\label{proof}
This section is dedicated to the proofs of our main results.
We first find an expression for
 the expected time spent by $(S,Y)$ until the ruin time for $X$ in the following lemma.

\begin{lem}\label{lem:2}
For $z> x, y\geq 0$, we have
\begin{equation*}
\begin{split}
R(x; dz,dy)&:=
\int_{0}^{\infty} \bP_{x}\big( S_{t}\in dz, Y_{t}\in dy; t<\tau_{0}^{-}\big)\,dt\\
&=\ \frac{W(x)}{W(z)} \Big(W(dy)- \frac{W'(z)}{W(z)} W(y)\,dy\Big) \,dz.
\end{split}
\end{equation*}
In particular,
\[R(x; dz, \{0\})= \frac{W(x)}{W(z)}W(0)\,dz.\]
\end{lem}
\begin{proof}[Proof of Lemma \ref{lem:2}]
Let $f$ be a nonnegative, continuous and bounded function on $[0,b]$. Then
\begin{equation*}
\begin{split}
 \int_{0}^{\infty} \bE_{x}\big(f(X_{t}); t<\tau_{0}^{-}, S_{t}<b\big)\,dt
&= \int_{0}^{\infty} \bE_{x}\big(f(X_{t}); t\leq \tau_{b}^{+}\wedge\tau_{0}^{-}\big)\,dt \\
&= \int_{0}^{b} f(y) \Big(\frac{W(x)}{W(b)}W(b-y)-W(x-y)\Big)\,dy.
\end{split}
\end{equation*}
Differentiating in $b$ on both sides of the above equation, we have for $z>x$ and $z>y\geq0$,
\begin{align*}
&\ \int_{0}^{\infty} \bP_{x}\big(S_{t}\in dz, X_{t}\in dy, t<\tau_{0}^{-}\big)\,dt\\
=&\ \frac{W(x)}{W(z)} \Big(W'(z-y)- \frac{W'(z)}{W(z)} W(z-y)\Big) \,dzdy
+ \frac{W(x)}{W(z)}W(0) \delta_{\{z\}}(dy)\,dz
\end{align*}
where $\delta_{\{z\}}(dy)$ denotes the Dirac measure at $\{z\}$.
The result of the lemma follows by change of variable.
\end{proof}
\begin{rmk}\label{rmk:cal:1} For $z>x\geq 0$ and $0<u\leq z$, by Lemma \ref{lem:2} we obtain the following density
\begin{align*}
&\ \frac{1}{dz}
\Big(\int_{0-}^{z} \omega(y) \wo(u,y) R(x;dz,dy)\Big) \\
=&\ \int_{0-}^{z}
\frac{W(x)}{W(z)} \wo(u,y)\omega(y)
\Big(W(dy)- \frac{W'(z)}{W(z)} W(y)\,dy\Big)\\
=&\ \frac{W(x)}{W(z)} \Big(-\wo_{2}(u)-W'(u)
- \frac{W'(z)}{W(z)} (\wo(u)-W(u))\Big),
\end{align*}
where we need identity \eqref{wo_2}  and
identity \eqref{eqn:defn:w:2} for the last equality.
In particular, for $u=z$
it becomes \[\D -\wo(z)\frac{W(x)}{W(z)}
\Big(\frac{W'(z)}{W(z)}+ \frac{\wo_{2}(z)}{\wo(z)} \Big).\]
\end{rmk}

To obtain the Feymann-Kac identity,
we need the following result from \cite{BoLi:sub1} and more properties of $\wo$.

\begin{prop}\label{prop:1}
For $x\in(c,b)$ we have
\[
\bE_{x}\Big(\exp\big(-\int_{0}^{\tau_{b}^{+}}
\omega(X_{s})\,ds\big); \tau_{b}^{+}<\tau_{c}^{-}\Big)
=\frac{\wo(x,c)}{\wo(b,c)}.
\]
In addition, the $\omega$-resolvent measure of $X$ killed at $\tau_{b}^{+}\wedge\tau_{c}^{-}$ is given by for $c<y<b$,
\begin{equation}
\begin{aligned}
&\ \int_{0}^{\infty}\bE_{x}\Big(
\exp\big(-\int_{0}^{t}\omega(X_{s})\,ds\big);  t\leq \tau_{b}^{+}\wedge\tau_{c}^{-}, X_{t}\in dy\Big)\,dt\\
&=\  \Big(\frac{\wo(x,c)}{\wo(b,c)}\wo(b,y)-\wo(x,y)\Big)\,dy.
\end{aligned}
\end{equation}
\end{prop}

\begin{lem}\label{lem:1}
Let $\omega$ be a nonnegative locally bounded measurable function,
$\wo$ and $W^{(\omega_{u})}$ be scale functions with respect to
$\omega$ and $\omega_{u}$, respectively,
where $\omega_{u}(z):=\omega(u-z)$ for some $u\in\mathbb{R}$,
that is
\begin{equation}\label{eqn:lem:wo}
W^{(\omega_{u})}(x,y)=W(x-y)
+\int_{y}^{x} W(x-z)\omega(u-z)W^{(\omega_{u})}(z,y)\,dz
\end{equation}
for $x,y\in\mathbb{R}$. Then \[W^{(\omega_{u})}(u-y,u-x)=\wo(x,y).\]
\end{lem}

\begin{proof}[Proof of Lemma \ref{lem:1}]
Denoting by $g(x,y):=W^{(\omega_{u})}(u-y,u-x)$
and by change of variable,
we have from \eqref{eqn:lem:wo} that
\begin{align*}
g(x,y)=&\ W(x-y)+
\int_{u-x}^{u-y}W(u-y-z)\omega(u-z) W^{(\omega_{u})}(z,u-x)\,dz\\
=&\ W(x-y)+
\int_{y}^{x}g(x,z) \omega(z)W(z-y)\,dz.
\end{align*}
It shows that $g$ satisfies the equation \eqref{eqn:defn:w:2}, which
 concludes the proof by the uniqueness of solution to \eqref{eqn:defn:w:2}.
\end{proof}

\begin{rmk}\label{rmk:2}
Combining  Lemma \ref{lem:1} and Proposition \ref{prop:1} gives
\[
\ \bE_{b-x}\Big(\exp\big(-\int_{0}^{\tau_{b}^{+}}\omega(b-X_{s})\,ds\big); \tau_{b}^{+}<\tau_{0}^{-}\Big)
=\frac{\wo(b,x)}{\wo(b,0)}\]
and
\begin{equation*}
\begin{split}
&\ \int_{0}^{\infty} \bE_{b-x}\Big(\exp\big(-\int_{0}^{t}\omega(b-X_{s})\,ds\big) f(b-X_{t});
t<\tau_{b}^{+}\wedge\tau_{0}^{-}\Big)\,dt\\
&=\ \int_{0}^{b}f(y)
\Big(\frac{\wo(b,x)}{\wo(b,0)}\wo(y,0)- \wo(y,x)\Big)\,dy,
\end{split}
\end{equation*}
which are the corresponding  results for the dual process of $X$.
\end{rmk}

Since both processes $X$ and $Y$ are involved in our problems,
we need to consider the joint probability law of processes $S$ and $Y$.
We slightly abuse notations in the following discussion.

Here, for some $u\in\mathbb{R}$ and $v\geq0$,
let $\D S_{t}:=u\vee \sup_{s\in[0,t]} X_{s}$ be the running maximum process of $X$,
and $Y:=S_{t}-X_{t}$ be the  process reflected from the running maximum.
The two-dimensional process
$Z\equiv\{Z_{t}, t\geq0\}:=
\{(S_{t},  Y_{t}),t\geq0\}$ is still a Markov process.
The law of $Z$ starting from $(u,v)$ is denoted by
$\bP_{u,v}:=\bP\big(\cdot\big|(S_{0}, Y_{0})=(u, v)\big)$
and the corresponding expectation is denoted by $\bE_{u,v}$.
We write $\bP_{x}:=\bP_{x,0}=\bP(\cdot|X_{0}=x)$ for $x\in\mathbb{R}$
which reduces to the  probability law of $X$ given $X_{0}=x$.

\begin{proof}[Proof of Proposition \ref{lem:3}]
 Define an operator on
the space of continuous and bounded functions
by
\[
\ro f(u,v):=
\int_{0}^{\infty} \bE\Big(e^{-L(t)} f(S_{t},Y_{t}); t\leq \tau_{b}^{+}\wedge\tau_{0}^{-}\Big|(S_{0},Y_{0})=(u,v)\Big)\,dt,
\]
for $(u,v)\in\mathbb{R}\times\mathbb{R}^{+}$.
Write
$$\ro f(u):=\ro f(u,0)\,\,\,\text{ and}\,\,\, Rf(u,v):=\ro f(u,v)\,\,\,\text{ for}\,\,\, \omega\equiv 0.$$
Then $Rf(u)=R^{(0)}f(u,0)$ whose expression can be found  in Lemma \ref{lem:2}.

Firstly, by the additivity of $L$, we have for every $t>0$,
\begin{equation}\label{eqn:feykac}
1- e^{-L(t)}=e^{-L(t)}\int_{0}^{t} e^{L(s)}\omega(Y_{s})\,ds
=\int_{0}^{\infty} \omega(Y_{s}) e^{-L(t-s)\circ \theta_{s}} \mathbf{1}_{(s<t)}\,ds,
\end{equation}
where $\theta_{\cdot}$ denotes the shift operator and $\mathbf{1}_{A}$ denotes the indicator function.
Plugging it into the following equation and
applying the Markov property at time $s>0$,
we have
\begin{equation}\label{eqn:ro:2}
\begin{split}
&Rf(x)-\ro f(x)\\
&=
\int_{0}^{\infty} \bE_{x,0}\Big(\big(1-e^{-L(t)}\big) f(S_{t},Y_{t}); t<\tau_{b}^{+}\wedge\tau_{0}^{-}\Big)\,dt\\
&=
\int_{0}^{\infty}\int_{0}^{\infty}
\bE_{x,0}\Big( \omega(Y_{s}) \big(e^{-L(t-s)} f(S_{t-s},Y_{t-s})\big)\circ\theta_{s}; s<t<\tau_{b}^{+}\wedge\tau_{0}^{-}\Big)\,ds\,dt\\
&=\ \int_{x}^{b} \int_{0-}^{z} R(x; dz,dy)\omega(y) \ro f(z,y)\\
&=: R(\omega \ro)f(x).
\end{split}
\end{equation}

On the other hand, under $\bP_{z,y}$ we have $X_{0}=z-y$.
The absence of positive jumps gives
$S_{\tau_{z}^{+}}=z$ on the set $\{\tau_{z}^{+}<\infty\}$
and $S_{t}=z$, $Y_{t}=z-X_{t}$ for $t<\tau_{z}^{+}$.
Conditioning on $\tau_{z}^{+}\wedge\tau_{0}^{-}$, we have
\begin{align*}
\ro&  f(z,y)
=\bE\Big(\exp\big(-\int_{0}^{\tau_{z}^{+}}\omega(z-X_{s})\,ds\big); \tau_{z}^{+}<\tau_{0}^{-}
\Big|X_{0}=z-y\Big)\cdot \ro f(z)\\
&\ + \int_{0}^{\infty}
\bE\Big(\exp\big(-\int_{0}^{t}\omega(z-X_{s})\,ds\big) f(z,z-X_{t}); t<\tau_{z}^{+}\wedge\tau_{0}^{-} \Big|X_{0}=z-y\Big)\,dt.
\end{align*}
By  the discussion in Remark \ref{rmk:2},  the identity above reduces to
\begin{equation}\label{eqn:ro:1}
\begin{aligned}
\ro f(z,y)=&\ \frac{\wo(z,y)}{\wo(z,0)}
\Big(\ro f(z)+ \int_{0}^{z} f(z,u) \wo(u,0)\,du\Big)\\
&\quad - \int_{0}^{z} f(z,u) \wo(u,y)\,du.
\end{aligned}
\end{equation}

Plugging it into the definition of $R(\omega \ro)f$  gives
\begin{align*}
R(\omega \ro)& f(x)=
 \int_{x}^{b} \int_{0-}^{z} R(x;dz,dy)\omega(y)
\Big(- \int_{0}^{z}f(z,u)\wo(u,y)\,du\\
&\quad+  \frac{\wo(z,y)}{\wo(z,0)}
\big(\ro f(z)+ \int_{0}^{z} f(z,u)\wo(u,0)\,du\big)\Big).
\end{align*}
By Remark \ref{rmk:cal:1}
and $\wo(z,0)=\wo(z)$, we further have
\begin{align*}
R&(\omega \ro) f(x)=
\int_{x}^{b}\int_{0}^{z} f(z,u) \frac{W(x)}{W(z)}
\Big(W'(u)- \frac{W'(z)}{W(z)}W(u)\Big)dudz\\
&\quad +
\int_{x}^{b}\int_{0}^{z} f(z,u) \frac{W(x)}{W(z)}
\Big(\wo_{2}(u)+ \frac{W'(z)}{W(z)}\wo(u)\Big)dudz\\
&\quad -\int_{x}^{b}
\Big(\ro f(z)+ \int_{0}^{z} f(z,u)\wo(u)\,du\Big)
\frac{W(x)}{W(z)}
\Big(\frac{\wo_{2}(z)}{\wo(z)}+ \frac{W'(z)}{W(z)}\Big)\,dz.
\end{align*}
Notice that  the resolvent density of $R(x;dz,du)$ for $u>0$ in Lemma \ref{lem:2} appears in the first integrand of the above equation.
Comparing it with the left hand side of \eqref{eqn:ro:2} and applying Lemma \ref{lem:2} gives the following integral equation on $\frac{\ro f(x)}{W(x)} $.
\begin{equation}\label{eqn:ro:3}
\begin{split}
 \frac{\ro f(x)}{W(x)}
&=  \frac{Rf(x)- R(\omega \ro)f(x)}{W(x)}\\
&= \int_{x}^{b}
\Big(\frac{W'(z)}{W(z)}+ \frac{\wo_{2}(z)}{\wo(z)} \Big)
\frac{\ro f(z)}{W(z)}\,dz
+ W(0) \int_{x}^{b}  \frac{f(z,0)}{W(z)} \,dz\\
&\quad+  \int_{x}^{b}\int_{0}^{z} \frac{f(z,u)}{W(z)}
 \Big( \frac{\wo_{2}(z)}{\wo(z)}\wo(u)-\wo_{2}(u) \Big)\,dudz.
\end{split}
\end{equation}

Recalling the definition of $\ho$ in \eqref{defn:ho},
one can check that
\[
\frac{\ho(x)}{W(x)}=
\frac{1}{W(1)}-\int_{1}^{x}\Big(\frac{W'(z)}{W(z)}+ \frac{\wo_{2}(z)}{\wo(z)} \Big)
\frac{\ho(z)}{W(z)}\,dz.
\]
Solving the equation \eqref{eqn:ro:3} for $\frac{\ro f(x)}{W(x)} $, we have that for $x\in[0,b]$,
\begin{align*}
\ro f(x)=&\ \int_{x}^{b}\int_{0}^{z} f(z,u)
\frac{\ho(x)}{\ho(z)}
\Big(\frac{\wo_{2}(z)}{\wo(z)}\wo(u)- \wo_{2}(u)\Big)\,dudz\\
&\quad + W(0) \int_{x}^{b} f(z,0) \frac{\ho(x)}{\ho(z)}\,dz,
\end{align*}
which gives the desired formula \eqref{eqn:thm1:ro+}.
\end{proof}

We are now ready to prove the  main results.

\begin{proof}[Proof of Theorem \ref{thm:1}]
We take use of the following version of the Feynman-Kac identity:
\begin{equation}
\label{eqn:feykac:2}
1- e^{-L(t)}=\int_{0}^{t} e^{-L(s)}\omega(Y_{s})\,ds
=\int_{0}^{\infty} \omega(Y_{s})  e^{-L(s)} \mathbf{1}_{(s<t)}\,ds.
\end{equation}
To simplify notations we denote  for $u\in[0,b]$ and $u\geq v\geq 0$
 $$
 \mathcal{A}^{(\omega)}(u,v):=
\bE\big(e^{-L(\tau_{b}^{+})}; \tau_{b}^{+}<\tau_{0}^{-}\big|(S_{0},Y_{0})=(u,v)\big),\\
$$
and write $\mathcal{A}^{(\omega)}(u,0):=\mathcal{A}^{(\omega)}(u)$.
It holds  that for $\omega\equiv0$
\begin{equation}\label{eqn:b+:1}
\mathcal{A}(u,v):=\mathcal{A}^{(0)}(u,v)=\bP_{u-v}\big(\tau_{b}^{+}<\tau_{0}^{-}\big)=\frac{W(u-v)}{W(b)}.
\end{equation}
Similar to \eqref{eqn:ro:2}, but with \eqref{eqn:feykac}  replaced by \eqref{eqn:feykac:2},
we have from the Markov property that
\begin{equation}\label{eqn:b+:2}
\mathcal{A}(x)- \mathcal{A}^{(\omega)}(x)
=\int_{x}^{b} \int_{0-}^{z} \ro(x;dz,dy) \omega(y) \mathcal{A}(z,y).
\end{equation}
Plugging \eqref{eqn:b+:1} and \eqref{eqn:thm1:ro+} into the equation gives
$$
\mathcal{A}^{(\omega)}(x)
=\frac{W(x)}{W(b)}+ \int_{x}^{b}
\frac{\ho(x)}{\ho(z)}
\Big( \frac{W'(z)}{W(b)}
+\frac{\wo_{2}(z)}{\wo(z)} \frac{W(z)}{W(b)}\Big)\,dz.
$$
Noticing that
$$\D \frac{\wo_{2}(z)}{\wo(z)}
=\frac{d}{dz}\Big(-\log \ho(z)\Big),$$
 we further have
$$
\mathcal{A}^{(\omega)}(x)=
 \frac{W(x)}{W(b)}+
\Big(\frac{\ho(x)}{\ho(z)} \frac{W(z)}{W(b)}\Big)\Big|_{z=x}^{z=b}
=\frac{\ho(x)}{\ho(b)},
$$
which proves the formula \eqref{eqn:thm1:b+}.

We proceed to show formula (\ref{eqn:thm1:c-}).
Observe first that
\begin{equation}
\label{eqn:c-:1}
\bE_{x,0}\left(\int_{0}^{\infty} e^{-L(t)}\omega(Y_{t}) \mathbf{1}_{\{ t<\tau_{b}^{+}\wedge\tau_{0}^{-}\}}\,dt\right)
=1- \bE_{x,0}\left(e^{-L(\tau_{b}^{+}\wedge\tau_{0}^{-})}\right).
\end{equation}
On the other hand,
by \eqref{eqn:thm1:ro+} and integration by parts  the left hand side of \eqref{eqn:c-:1} is equal to
\begin{equation*}
\begin{split}
 &
 \int_{x}^{b} \frac{\ho(x)}{\ho(z)}
 \Big(\frac{\wo_{2}(z)}{\wo(z)}
 \big(\zoh(z)-1\big)- \zoh_{2}(z) \Big)\,dz\\
&=\ \frac{\ho(x)}{\ho(z)}(\zoh(z)-1)\Big|_{x}^{b}
- \int_{x}^{b} \frac{\ho(x)}{\ho(z)}(\zoh_{1}(z)+\zoh_{2}(z))\,dz\\
&=\ 1- \frac{\ho(x)}{\ho(b)}- \Big(\zoh(x)- \frac{\ho(x)}{\ho(b)}\zoh(b)\Big)\\
&\quad\ - \int_{x}^{b} \frac{\ho(x)}{\ho(z)}\big(\zoh_{1}(z)+\zoh_{2}(z)\big)\,dz.
\end{split}
\end{equation*}
Comparing with (\ref{eqn:c-:1}) gives the formula (\ref{eqn:thm1:c-}).
We thus completes the proof of Theorem \ref{thm:1}.
\end{proof}

\begin{proof}[Proof of Proposition \ref{prop:3}]
Observing that
$$\D \delta t+ L(t)=\int_{0}^{t}(\delta+\omega(X_{s}))\,ds$$ for every $t>0$,
the result for general case of $\delta>0$
can be derived from that of the special case $\delta=0$,
with the weight function replaced by $\omega+\delta$.
We thus focus on the case $\delta=0$ in the proof.

To this end, we first evaluate,
for a continuous $f\geq0$ on $[0,b]$,
\begin{equation}\label{eqn:p3:1}
\begin{split}
&\ \int_{0}^{\infty} \bE_{x}\big(e^{-L(t)}f(X_{t}); t<\tau_{b}^{+}\wedge \tau_{0}^{-}\big)\,dt\\
=&\ \int_{x}^{b}f(z-y) \,dz \int_{0}^{z}
\frac{\ho(x)}{\ho(z)} \Big(\frac{\wo_{2}(z)}{\wo(z)}\wo(y)-\wo_{2}(y)\Big)\,dy\\
&\quad + W(0)\int_{x}^{b}f(z)\frac{\ho(x)}{\ho(z)}\,dz
\end{split}
\end{equation}
with Proposition \ref{lem:3} applied.
Changing the order of integrals, the first term of the right hand side of \eqref{eqn:p3:1} equals to
\begin{align*}
&\ \int_{x}^{b}\,dz \int_{0}^{z} f(y)
\frac{\ho(x)}{\ho(z)} \Big(\frac{\wo_{2}(z)}{\wo(z)}\wo(z-y)-\wo_{2}(z-y)\Big)\,dy\\
=&\ \int_{0}^{b} f(y)\,dy \int_{x\vee y}^{b}
\frac{\ho(x)}{\ho(z)} \Big(\frac{\wo_{2}(z)}{\wo(z)}\wo(z-y)-\wo_{2}(z-y)\Big)\,dz.
\end{align*}
Noticing that $$\frac{d}{dx}\ho(x)=-\ho(x)\frac{\wo_{2}(x)}{\wo(x)},$$ we further have
\begin{align*}
&\ \int_{x\vee y}^{b}
\frac{\ho(x)}{\ho(z)} \Big(\frac{\wo_{2}(z)}{\wo(z)}\wo(z-y)-\wo_{2}(z-y)\Big)\,dz\\
=&\ \frac{\ho(x)}{\ho(z)}\wo(z-y)\Big|_{z=x\vee y}^{z=b}
- \int_{x\vee y}^{b} \frac{\ho(x)}{\ho(z)}\big(\wo_{1}(z-y)+\wo_{2}(z-y)\big)\,dz.
\end{align*}
Since $\wo(z-y)=0$ for $z<y$, we  have
\[
\frac{\ho(x)}{\ho(z)}\wo(z-y)\Big|_{z=x\vee y}^{z=b}
=\frac{\ho(x)}{\ho(b)}\wo(b-y)-\wo(x-y)- W(0)\frac{\ho(x)}{\ho(y)}\mathbf{1}(y>x).
\]
Plugging into \eqref{eqn:p3:1} and using the fact that $\wo_{1}(z-y)=\wo_{2}(z-y)=0$ for $z<y$,
we further have
\begin{equation}
\begin{split}
&\ \int_{0}^{\infty} \bE_{x}\big(e^{-L(t)}f(X_{t}); t<\tau_{b}^{+}\wedge \tau_{0}^{-}\big)\,dt\\
=&\ \int_{0}^{b}
\Big(\frac{\ho(x)}{\ho(b)}\wo(b-y)-\wo(x-y)\Big) f(y)\,dy\\
&\quad - \int_{0}^{b}f(y)
\Big(\int_{x}^{b} \frac{\ho(x)}{\ho(z)}\big(\wo_{1}(z-y)+\wo_{2}(x-y)\big)\,dz\Big)\,dy.
\end{split}
\end{equation}
The formula \eqref{eqn:prop3:1} then follows from the compensating formula.
\end{proof}

For the proof of Proposition \ref{prop:2},
we apply \eqref{eqn:defn:w:2} for $\wo$, and definition
\eqref{eqn:defn:zh} for $\zoh$. We also need  the following identity
which first appears in \cite{Loeffen2014:occupationtime:levy}.
\begin{equation}
\label{eqn:scaleiden}
W^{(q)}(x-y)=W(x-y)+ q \int_{y}^{x}W(x-z)W^{(q)}(z-y)\,dz.
\end{equation}
\begin{proof}[Proof of Proposition \ref{prop:2}]
Firstly, we have from \eqref{eqn:defn:zh} for $\zoh$ and the Fubini's theorem that
\begin{align}
&\ \int_{y}^{x}\zoh(x,u)\omega(u)W(u-y)\,du\non\\
=&\  \int_{y}^{x}\omega(u)W(u-y)\,du
+ \iint_{x>u>v>y}\omega(u)\wo(u,v)\omega(v)W(v-y)\,dudv\non\\
=&\  \int_{y}^{x}\omega(u)
\Big(W(u-y)+ \int_{y}^{u}\wo(u,v)\omega(v)W(v-y)\,dv\Big)\,du\non\\
=&\ \int_{y}^{x}\omega(u)\wo(u,y)\,du=\zoh(x,y)-1\label{eqn:prop:1}
\end{align}
where identity \eqref{eqn:defn:w:2} for $\wo$ is used for the second to the last equation.

On the other hand, for any $q\geq0$,
integrating on both sides of the equation above with respect to $W^{(q)}(y-c)\,dy$ over interval $(c,x)$,
by the scale function identity \eqref{eqn:scaleiden} we have
\begin{align}
&\ q\int_{c}^{x}\big(\zoh(x,y)-1\big)W^{(q)}(y-c)\,dy\non\\
=&\ q\int_{c}^{x} \big(\int_{y}^{x} \zoh(x,u)\omega(u)W(u-y)\,du\big) W^{(q)}(y-c)\,dy\non\\
=&\ q\int_{c}^{x} \zoh(x,u)\omega(u) \int_{c}^{u}W(u-y)W^{(q)}(y-c)\,dy\non\\
=&\ \int_{c}^{x}\zoh(x,u)\omega(u) (W^{(q)}(u-c)-W(u-c))\,du.\label{eqn:prop:2}
\end{align}
Combining \eqref{eqn:prop:1}, \eqref{eqn:prop:2}  and
\[Z^{(q)}(x-y)=1+ q \int_{y}^{x} W^{(q)}(z-y)\,dz,\]
 we have
\begin{align}
\zoh(x,y)-Z^{(q)}(x-y)=\int_{y}^{x}\zoh(x,z)(\omega(z)-q)W^{(q)}(z-y)\,dz.\label{eqn:prop:4}
\end{align}

We now consider the case that $\omega$ is a step function given by (\ref{step_function}).
Given $\widehat{Z}_{n}$ on $\mathbb{R}\times\mathbb{R}$,
 since $\omega_{n}(y)=\omega_{n+1}(y)$ for  $y\geq a_{n+1}$,
it follows from \eqref{eqn:prop:1} that  $\widehat{Z}_{n+1}(x,y)= \widehat{Z}_{n}(x,y)$.
For $y<a_{n+1}$, taking $q=p_{n+1}$ in the equation \eqref{eqn:prop:4},
and noticing that $\D \omega_{n+1}(z)-p_{n+1}=0$ for $z< a_{n+1}$
and $\D \omega_{n+1}(z)=\omega_{n}(z)$ for $z\geq a_{n+1}$,
we have
\begin{equation*}
\begin{split}
\widehat{Z}_{n+1} (x,y)
&=\ \widehat{Z}^{(p_{n+1})}(x-y)+ \int_{y}^{x} \widehat{Z}_{n+1}(x,z)(\omega_{n+1}(z)-p_{n+1})W^{(p_{n+1})}(z-y)\,dz\\
&=\ \widehat{Z}^{(p_{n+1})}(x-y)+\int_{a_{n+1}}^{x} \widehat{Z}_{n+1}(x,z)(\omega_{n+1}(z)-p_{n+1}) W^{(p_{n+1})}(z-y)\,dz\\
&=\ \widehat{Z}^{(p_{n+1})}(x-y)+\int_{a_{n+1}}^{x} \widehat{Z}_{n}(x,z)(\omega_{n}(z)-p_{n+1}) W^{(p_{n+1})}(z-y)\,dz\\
&=\ \widehat{Z}_{n}(x,y)- \int_{y}^{a_{n+1}} \widehat{Z}_{n}(x,z)(\omega_{n}(z)-p_{n+1})W^{(p_{n+1})}(z-y)\,dz\\
&=\ \widehat{Z}_{n}(x,y)+ (p_{n+1}-p_{n})\int_{y}^{a_{n+1}}\widehat{Z}_{n}(x,z)W^{(p_{n+1})}(z-y)\,dz,
\end{split}
\end{equation*}
where identity \eqref{eqn:prop:4} is used for the first and the fourth equation.
The proof is thus completed.

The recursive expression for $W_n(x,y)$ can be proved similarly.
\end{proof}

\medskip

\noindent
{\bf Acknowledgement} Yun Hua and Bo Li thank Concordia University where the first draft of this paper was completed during their visit in early 2017.

\bibliographystyle{apalike}

\end{document}